\documentclass[11pt]{article}  
\usepackage[margin=1in]{geometry}

\usepackage{amsthm, amsmath, amsfonts, amssymb, bbm} 
\usepackage{graphicx} 
\usepackage{mathrsfs} 
\usepackage{enumitem} 
\usepackage{float}
\usepackage{caption} 
\usepackage[normalem]{ulem} 
\usepackage{array}
\usepackage{xcolor} 
\graphicspath{{Figures/}} 

\theoremstyle{theorem}
\newtheorem{theorem}{Theorem}
\theoremstyle{lemma}
\newtheorem{lemma}{Lemma}[section]
\theoremstyle{definition}
\newtheorem{defn}{Definition}[section]
\newtheorem{prop}{Proposition}[section]
\theoremstyle{remark}

\DeclareMathOperator\variance{Var}

\newcommand{\mc}[1]{\mathcal #1}
\newcommand{\mb}[1]{\mathbf #1}
\newcommand{\mbb}[1]{\mathbb #1}
\newcommand{\prob}{\mathbb{P}}
\newcommand{\deriv}[1]{\partial_{#1}}

\newcommand{\ceiling}[1]{\lceil #1 \rceil}
\newcommand{\floor}[1]{\lfloor #1 \rfloor}

\newcommand{\nosquare}{\let\qed\relax}	
\newcommand{\eps}{\varepsilon}

\newcommand{\niceBullet}{\raisebox{0.6pt}{\scalebox{0.55}{\textbullet}}\hspace{5pt}}

\newcommand{\canonicalPolygon}			{\mc P}

\newcommand{\treeClassRootShared}			{\mc T^R}
\newcommand{\treeClassFirstLeafShared}		{\mc T^L}
\newcommand{\treeClassNeitherShared}		{\mc T^N}

\newcommand{\genFunRootShared}			{ T^R }
\newcommand{\genFunFirstLeafShared}		{ T^L }
\newcommand{\genFunNeitherShared}		{ T^N }

\newcommand{\deltaDomain}			{\delta(\eps, \phi)}
\newcommand{\rDomain}				{\mc R(\eps, \phi)}
\newcommand{\rmDomain}				{\mc R^m(\eps, \phi)}
\newcommand{\myRdomain}				{\mc R(\eps_0, \phi_0)}
\newcommand{\scaledTstar}			{\tau}

\newcolumntype{L}[1]{>{\vspace{1mm}\raggedright\let\newline\\\arraybackslash\hspace{0pt}}m{#1}}
\newcolumntype{C}[1]{>{\centering\let\newline\\\arraybackslash\hspace{0pt}}m{#1}}

\makeatletter
\renewcommand\section{\@startsection{section}{1}{\z@}%
                                   {-3ex \@plus -1ex \@minus -.2ex}%
                                   {2ex \@plus.2ex}%
                                   {\large\bfseries}}
\makeatother

\begin{document}

\title{\vspace{-1.5cm} 		
			A concentration inequality for the maximum 
			\\ vertex degree in random dissections 
}

\author	{
			Kelvin Rivera-Lopez\footnote{
				Universit\'{e} de Lorraine, CNRS, IECL, F-54000 Nancy, France,
				kelvin.rivera-lopez@univ-lorraine.fr
			} 
			\and 
			Douglas Rizzolo\footnote{
				University of Delaware, Department of Mathematical Sciences, Newark, DE, USA,
				drizzolo@udel.edu
			}
		}

\maketitle

\vspace{-8mm}

\begin{abstract}
	We obtain a concentration inequality for the maximum degree of a vertex in a uniformly random dissection of a polygon. 
	This resolves a conjecture posed by Curien and Kortchemski in 2012.
	Our approach is based on a bijection with dual trees and the tools of analytic combinatorics.
\end{abstract}


\section{Introduction}

In this paper we establish a concentration inequality for the maximum degree of a vertex in a uniformly random dissection of a polygon. This problem has been considered before, with a concentration inequality being established in \cite{bernasconi10}, and improved in \cite{curien14}.  However, the authors of \cite{curien14} believed that their result was not optimal and conjectured that if $\Delta_n$ is the maximum vertex degree in a uniformly random dissection of an $n$-gon, then
$$
	\prob(
			\big|
					\Delta_n - ( \log_b(n) + \log_b \log_b(n) ) 
			\big| 
		> 	c \log_b \log_b(n)
		) 	
			\longrightarrow 
				0,
$$
as $ n \to \infty $, where $b = 1 + \sqrt{2}$ and $c$ is an arbitrary positive number.  Our main contribution is proving that this conjecture is correct.  More precisely, we prove the following result.

\begin{theorem}
\label{mainResult}
	Let $b = 1 + \sqrt 2$, $\lambda_n = \log_b n + \log_b\log_b n $, $\Omega_n \to \infty$,
	and $ \Delta_n $ denote the maximum vertex degree in a random dissection of an $n$-gon.
	Then, as $ n \to \infty $,
	$$
		\prob(|\Delta_n - \lambda_n| \ge \Omega_n)
			=
				O( ( \log n )^{ -1 } + b^{-\Omega_n})
			.
	$$
	
\end{theorem}

In \cite{gao00}, an analogous result for random triangulations was obtained using analytic combinatorics to analyze generating functions.
We use the same approach here but also use a bijection from dissections to their dual trees to make the combinatorics tractable.
Dual trees have been used previously to study random dissections in \cite{curien14} and random triangulations in \cite{devroye1999properties}.  The key observation in our case is that, under this bijection, vertex degrees become the distance between consecutive leaves.  In the case of triangulations this observation was made in \cite{devroye1999properties}.

The paper is organized as follows. In Section 2, we reduce the proof of Theorem \ref{mainResult} to establishing some estimates on the moments of the number of vertices of a fixed degree.
The remainder of the paper is dedicated to obtaining these estimates. 
In Section 3, we introduce the bijection between dissections and their dual trees. 
In Section 4, we study the generating functions of certain classes of trees needed to establish the moment estimates in Section 2.
In Section 5, we employ the techniques of analytic combinatorics to obtain the estimates of Section 2.

\section{Dissections}
In this section, we reduce the study of the maximum vertex degree to one of the number of vertices of a fixed degree. 
We begin by introducing the objects of interest.

\begin{defn}
	For $n \ge 3$, let $\canonicalPolygon_n$ denote the convex $n$-gon in the plane whose vertices are the $n^{th}$ roots of unity. 
	A subset of the plane $d$ is a \emph{dissection} of $\canonicalPolygon_n$ if it is the union of the sides of $\canonicalPolygon_n$ and a collection of diagonals that may intersect only at their endpoints. 
	In this case, we define for $ j = 1, 2, ..., n $, the $j^{th}$ vertex of $d$ as the point $v_j(d) = e^{2 \pi i (n+1-j)/n}$, and the degree of $v_j(d)$, denoted by deg $v_j(d)$, is the total number of diagonals and sides of $\canonicalPolygon_n$ that lie in $d$ and contain $v_j(d)$.
	See Figure \ref{dissectionFigure}.
	For convenience, we will often omit the argument of $ v_j(d) $ when it is clear from context.
	
\end{defn}

\begin{figure}[t]
    \centering
	\def\svgwidth{0.4\textwidth} 
	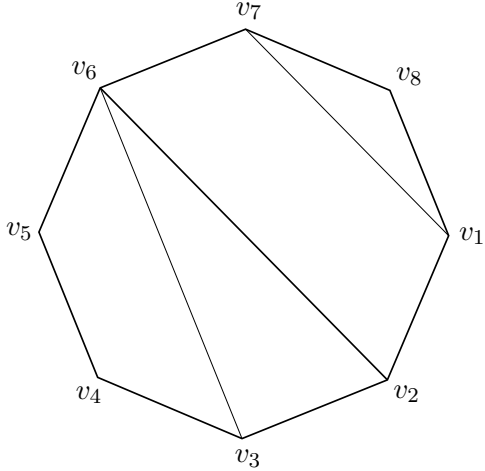
	\caption{A dissection of an octagon.}
	\label{dissectionFigure}
\end{figure}

We denote the set of all dissections of $\canonicalPolygon_{n+1}$ by $\mc D_n$. The subset of these in which $v_1$ has degree $k$ will be denoted by $\mc D_{n, k}$. The collection of pairs $(d, v)$ in which $d$ is a dissection from $\mc D_{n, k}$ and $v \neq v_1$ is a vertex in $d$ with degree $l$ will be denoted by  $\mc D_{n, k, l}$. We will refer to such a pair as a dissection with a \textit{distinguished vertex} $v$. The parameters $n$, $k$, and $l$ are to take integer values no less than 2.

For each $n$, we construct a probability space by equipping $\mc D_n$ with the uniform measure.  
On each space, we define a random variable $\Delta_n$ that maps a dissection $d$ to the maximum vertex degree of $d$. 
In addition, we define a sequence of random variables $\{\zeta_{k, n}\}_{k \ge 2}$ by letting $\zeta_{k, n}$ (or $\zeta_k$ for short) map a dissection $d$ to the number of vertices in $d$ having degree $k$. 
Theorem \ref{mainResult} is a direct consequence of the following lemma.

\begin{lemma}
\label{momentEstimates}
Let $b = 1 + \sqrt 2$. The following estimates hold:
\begin{enumerate}[label = (\roman*)]
	\item
	\label{first moment estimate, all k}
	for every $p \in (0, b)$ there exists some $ M > 0 $ such that
	$$
		\mbb{E}(\zeta_k) 	
			\le 
				M n p^{-k}
	$$
	
	\noindent
	for $ k \ge 2 $ and large $ n $,	
	\item
	\label{first moment estimate, log k}
	for every $ L > 0 $ there exists some $ M > 0 $ such that
	$$
		\left|
			\frac{
        			b^{ k }
        			\,
        			\mbb{E}(\zeta_k) 
				}{
        			2 n
				}
		-
			k
		\right|	
			\le
				M
	$$

	\noindent
	whenever $ k \le L \log n $ and $ n $ and $ k $ are large,
	
	\item
	\label{second moment estimate}
	for every $ L > 0 $ there exists some $ M > 0 $ such that
	$$
		k 
		\left|
    		\frac{ \mbb{E}(\zeta_k(\zeta_k - 1)) }
    		{\mbb{E}(\zeta_k)^{ 2 }}
		-
			1
		\right|
			\le	
				M
	$$
	
	\noindent
	whenever $ k \le L \log n $ and $ n $ and $ k $ are large.
\end{enumerate}

\end{lemma}

\begin{proof}[Proof of Theorem 1]

    It suffices to consider the case $\Omega_n = O(\log \log n)$ (for the general case, apply the result to 
    $
    	\Omega'_n 
			= 
				\min(\Omega_n, \log_b\log n)
	$). 
    Setting
    $
     	k_n = \ceiling{\lambda_n} - \floor{\Omega_n},
    $ 
    it follows immediately that
    $$
    	C \log n
			\le
				\lambda_n - \Omega_n
			\le
				k_n
			\le
				\lambda_n - \Omega_n + 2
			\le
				L \log n
    $$

	\noindent
	for some constants $ C, L > 0 $ and large $ n $.
	Let $ M > 0 $ be the maximum of the constants obtained by applying \ref{first moment estimate, log k} and \ref{second moment estimate} with the constant $ L $.
	Using the bound in \ref{first moment estimate, log k}, the inequality $ b = 1 + \sqrt{ 2 } > 1 $, and the identity
	$
		b^{ \lambda_n }
			=
				n \log_b n
			,
	$
	we obtain a lower bound for the sequence
    $
    	\mu_n 
			= 
				\mbb{E}(\zeta_{k_n})
    $:
    for large $ n $,
    \begin{align*}
    	\mu_n 	
    		& \ge
				( k_n - M )
				2 n
				b^{-k_n}
			\\
    		& \ge
				k_n
				n
				b^{- \lambda_n + \Omega_n - 2}  
			\\
    		& \ge
				( C \log n )
				n
				( n \log_b n )^{ -1 }
				b^{ \Omega_n - 2 }
			\\
    		& =
			 	C
				\log (b)
				\,
				b^{ \Omega_n - 2}
			.
    \end{align*}
    
    \noindent 
    We combine this with the bound in \ref{second moment estimate} as follows: for large $ n $,
    \begin{align*}
		\frac{ 
				\variance(\zeta_{k_n}) 
			}{
				\mu_n^{ 2 }
			}
				& = 	
                		\left|
                			\frac{ 
            							\mbb{E}( \zeta_{k_n}(\zeta_{k_n} - 1) ) 
            						- 	\mbb{E}(\zeta_{k_n})^2 
            						+ 	\mbb{E}(\zeta_{k_n})
                				}{
                					\mu_n^{ 2 }
                				}
                		\right|
				\\
				& \le 	
                		\left|
                			\frac{ 
            							\mbb{E}( \zeta_{k_n}(\zeta_{k_n} - 1) ) 
                				}{
                					\mu_n^{ 2 }
                				}
							-
								1
                		\right|
						+
                		\left|
        						\mu_n^{ -1 }
                		\right|
				\\
				& \le 	
        			 	M k_n^{ -1 }
					+
        			 	( C \log b )^{ -1 }
        				\,
        				b^{ 2 - \Omega_n }
				\\
				& = 	
						O( ( \log n )^{ -1 } + b^{-\Omega_n} )
				.
    \end{align*}

    \noindent 
    Applying Markov's Inequality then gives us the bound for the left tail: 
    as $ n \to \infty $,
    \begin{align*}
    	\prob(\Delta_n < \lambda_n - \Omega_n) 	
    			& \le 	
    					\prob\left(\sum_{k \ge \lambda_n - \Omega_n} \zeta_k = 0 \right) 
    			\\
    			& \le 	
    					\prob\left(\zeta_{k_n} = 0 \right) 
    			\\
    			& \le 	
    					\prob( ( \zeta_{k_n} - \mu_n )^2 \ge \mu_n^2) 
    			\\
    			& \le 	
    					\frac{\variance(\zeta_{k_n})}{\mu_n^2} 
    			\\
    			& = 	
    					O( ( \log n )^{ -1 } + b^{-\Omega_n})
				.
    \end{align*}

    Now let 
    $
    	\mu'_k
			= 
				\mbb{E}(\zeta_{k})
			,
    $
	$ p \in ( 1, b ) $, $L' \ge 2 \, (\log p)^{-1} $, and $ M' > 0 $ be the maximum of the constants obtained by applying \ref{first moment estimate, all k} and \ref{first moment estimate, log k} with the constant $ L' $.
	Using \ref{first moment estimate, all k} and \ref{first moment estimate, log k}, 
	the inequality $ 1 < p < b $,
	the bound $ p^{ - L' \log n } = n^{ - L' \log p } \le n^{ -2 } $,
    and 
    the identity $ b^{ \lambda_n } = n \log_b n $,
    we obtain the following bound:
    for large $ n $,
    \begin{align*}
    	\sum_{k \ge \ceiling{ \lambda_n + \Omega_n } }
			\mu'_k 	
				& \le 	
						\sum_{k \ge \ceiling{ \lambda_n + \Omega_n } }^{ \floor{ L' \log n } }
							\mu'_k 
						+ 	\sum_{k \ge \ceiling{ L' \log n } } 
							\mu'_k 
				\\
    			& \le 	
						\sum_{k \ge \ceiling{ \lambda_n + \Omega_n } }^{ \floor{ L' \log n } }
							2 n 
							( M' + k ) 
							b^{ -k }
					+ 	\sum_{k \ge \ceiling{ L' \log n } } 
							M' n 
							p^{-k}
				\\
    			& \le 	
						4 L' 
						\,
						( n \log n )
						\sum_{k \ge \ceiling{ \lambda_n + \Omega_n } }
							\!
							\!
							\!
							\!
							\!
							b^{ -k }
					+ 	M' n
						\sum_{k \ge \ceiling{ L' \log n } } 
							\!
							\!
							\!
							\!
							\!
							p^{-k}
				\\
    			& = 	
						\frac{ 
								4 L'
							}{
								1 - b^{ -1 }
							}
						( n \log n )
						\,
						b^{ - \ceiling{ \lambda_n + \Omega_n } }
					+ 	
						\frac{ 
								M' 
							}{
								1 - p^{ -1 }
							}
						n
						p^{-\ceiling{ L' \log n }}
				\\
    			& \le 	
						\frac{ 
								4 L' 
							}{
								1 - b^{ -1 }
							}						
						( n \log n )
						\,
						b^{ - \lambda_n - \Omega_n }
					+ 	
						\frac{ 
								M'
							}{
								1 - p^{ -1 }
							}
						n
						p^{- L' \log n }
				\\
    			& \le 	
						\frac{ 
								4 L' 
								\log b
							}{
								1 - b^{ -1 }
							}
						\,
						b^{ - \Omega_n }
					+ 	
						\frac{ 
								M'
							}{
								1 - p^{ -1 }
							}
						n^{ -1 }
				\\
    			& = 	
    					O( b^{-\Omega_n} + ( \log n )^{ -1 } )
    \end{align*}

    \noindent 
	Applying Markov's Inequality then gives us the bound for the right tail: as $ n \to \infty $,
    \begin{align*}
    	\prob(\Delta_n > \lambda_n + \Omega_n) 	
				& \le 	
						\prob 	\left(
										\sum_{k > \lambda_n + \Omega_n} 
											\zeta_k 
												\ge 
													1 
								\right) 
				\\
    			& \le 	
						\sum_{k > \lambda_n + \Omega_n} 
							\mu'_k
    			\\
				& = 	
    					O( b^{-\Omega_n} + ( \log n )^{ -1 } )
				.
    \end{align*}
 
    \noindent
    This concludes the proof.       
\end{proof}

\section{Dual Trees}
In this section, we make the connection between the factorial moments of the $\zeta_k$ and the enumeration of trees. 
The first step towards this goal is expressing these moments in terms of the sizes of our dissection classes. This is given in the following result.

\begin{lemma}
\label{momentsAndDissectionClasses}

For $ n, k \ge 2 $, let $ D_n $ and $ D_{ n, k } $ denote the size of 
$ \mc D_n $ and $ \mc D_{ n, k } $ respectively.
The following identities hold:
\begin{enumerate}[label = (\roman*)]
	\item
	\label{first moment identity}
	$\mbb{E}(\zeta_k) 				= 	(n+1)\frac{D_{n, k}}{D_n}$, 
	
	\item
	\label{second moment identity}
	$\mbb{E}(\zeta_k (\zeta_k-1)) 	= 	(n+1)\frac{D_{n, k, k}}{D_n}$.

\end{enumerate}
\end{lemma}

\begin{proof}
    Let $ \mc D^{ (i) }_n $ be the subset of $ \mc D_n $ consisting of dissections with $ i $ vertices of degree $ k $. 
    Let $ \mc D^{ (i) }_{ n, k, k } $ be the subset of $ \mc D_{ n, k, k } $ consisting of pairs $ ( d, v ) $ with 
    $
    	d 
    		\in 
    			\mc D^{ (i) }_n
    		.
    $
    Consider now a pair $(d, (u, v))$ where 
    $
    	d 
    		\in 
    			\mc D^{ (i) }_n
    $
    and $ ( u, v ) $ is an ordered pair of vertices in $d$ with degree $k$. 
    It should be clear that the number of such pairs is given by 
    $
    	D^{ (i) }_n 
    	i ( i - 1 )
    	,
    $
    where 
    $
    	D^{ (i) }_n 
    $
    is the size of $ \mc D^{ (i) }_n $.
    On the other hand, any such pair can be constructed by taking some
    $
    	( d, v_j ) 
    		\in
    			\mc D^{ (i) }_{ n, k, k }
    		,
    $
    distinguishing $ v_1 $ to obtain
    $
    	( d, ( v_1, v_j) ),
    $
    and rotating by a map of the form $z \mapsto z e^{-2\pi i m/(n+1)}$ for some $0 \le m \le n$. 
    Therefore, we have the identity 
    $$
    	D^{(i)}_{n} i (i-1)
    		= 	
    			D^{(i)}_{n, k, k} (n + 1)
    		,
    			\qquad
    			i \ge 0
    		,
    $$ 
    
    \noindent 
    where 
    $
    	D^{(i)}_{n, k, k}
    $
    is the size of $ \mc D^{(i)}_{n, k, k} $.
    Summing this relation over $i$ and dividing by $D_n$ yields the claim in \ref{second moment identity}. 
    The claim in \ref{first moment identity} can be obtained in a similar manner. 
\end{proof}

Our next step is to establish the connection between dissections and their dual trees.
To begin, let us define the types of trees we will be working with.

\begin{defn}
    A \emph{rooted ordered tree} is a graph-theoretic tree with a vertex designated as the \emph{root} and an ordering among the children of any vertex. 
    In such a tree, a non-root vertex is a \emph{leaf} if it has no children.
\end{defn}

We will think of rooted ordered trees as subsets of the upper half of the complex plane by embedding them in such a way that the root is mapped to $z = 0$ and the clockwise orientation agrees with the ordering throughout the tree. 
Up to orientation-preserving homeomorphisms of the upper half plane, a rooted ordered tree then corresponds to a unique subset of the plane. 
We denote the root of a tree $ t $ with $n$ leaves by $\ell_0(t) = \ell_{n+1}(t)$ and its leaves by $\ell_1(t), ..., \ell_n(t)$ (in clockwise order).
As with the vertices of a dissection, we will often omit the argument $ t $ when it is clear from context.

\begin{figure}[t]
    \centering
	\def\svgwidth{0.75\textwidth} 
	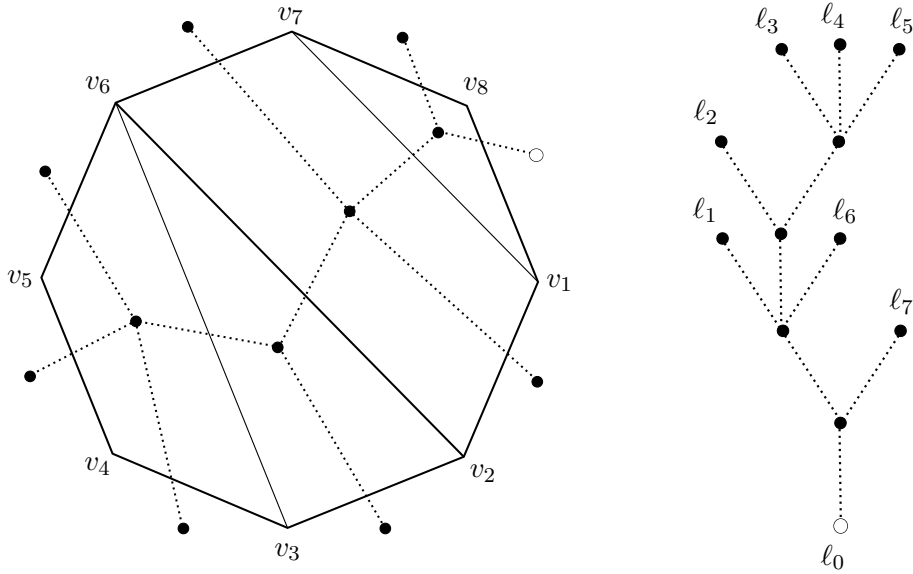
	\caption{A dual tree. Its root is colored white.}
	\label{dualTreeFigure}
\end{figure}

For each dissection $d \in \mc D_n$, we construct its dual tree, $t_d$, in the following way: 
	first, we place a vertex in each inner face of $d$ and connect those vertices whose corresponding faces share an edge; 
	then, we place $n$ vertices in the outer face, assign each to a distinct edge of $\canonicalPolygon_{n+1}$, and connect each to the vertex whose corresponding face shares the assigned edge; 
	finally, we root the tree at the vertex assigned to the edge $(v_1, v_{n+1})$ by applying an orientation-preserving homeomorphism of the plane that maps this vertex to $z = 0$ and the remainder of the tree into the upper half plane (see Figure \ref{dualTreeFigure}). 
In addition, we correspond a dissection with a distinguished vertex to a tree with a distinguished leaf by designating $(t_d, \ell_{j-1})$ as the dual `tree' of $(d, v_j)$.

As in \cite{curien14}, the map sending a dissection to its dual tree is bijective. The image of $\mc D_n$, which we will denote by $\mc T_n$, consists of rooted ordered trees that have $n$ leaves, root degree one, and no non-root vertex with exactly one child. To determine the images of the other classes, we use the following result.

\begin{prop}
\label{degAndLeftmostPaths}

Let $d \in \mc D_n$ and $\rho$ be the graph metric on $t_d$. The following identity holds:
$$
	\deg v_i 
		= 
			\rho(\ell_{i-1}, \ell_i)
		,
			\qquad
			i = 1, ..., n+1
		.
$$
\end{prop}

\begin{proof}

Superimpose $t_d$ onto $d$ (as in Figure \ref{dualTreeFigure}) and assign to each edge in $d$ the unique edge in $t_d$ that it intersects. This assignment maps the edges adjacent to $v_i$ to the edges in $t_d$ that constitute the path from $\ell_{i-1}$ to $\ell_i$. Thus, these two groups of edges are equal in number.
\end{proof}

Proposition \ref{degAndLeftmostPaths} reveals that the image of $\mc D_{n, k}$ under the duality map is the subset of $\mc T_n$ of trees in which the path from the root to the leftmost leaf is of length $k$ (contains $k$ edges). 
We denote this set by $\mc T_{n, k}$. 
The image of $\mc D_{n, k, l}$, which we denote by $\mc T_{n, k, l}$, consists of pairs $(t, \ell_i)$ in which $t$ lies in $\mc T_{n, k}$ and the path from the distinguished leaf $\ell_i$ to $\ell_{i+1}$ has length $l$. 
We will refer to this path as the \textit{$l$-path} of such a tree.

\section{Functional Equations}
\label{sectionFunctionalEquations}
In this section, we derive functional equations involving the generating functions that correspond to our tree classes. 
Our arguments will be purely combinatorial and will require us to consider some additional kinds of trees. 
As with our dual trees, we organize these new trees into sets that are indexed by some tree parameter(s).
These sets are described in Table \ref{table describing trees}.
We also include the sets containing our dual trees so that this table is a complete reference for all of the trees we will encounter.

\begin{table}[t]
	\centering
    \begin{tabular}{ C{1.3cm} | C{2cm} | L{9cm} }
    	\hline
		Set of Trees
    		& 	
    			Parameter Range						
    		& 	
    			Description of an element
    		\\ \hline
    	$ \mc T^*_n $
    		& 	
    													$ n \ge 1  $ 					
    		& 
    				\niceBullet rooted ordered tree with $ n $ leaves 
    				\newline
        			\niceBullet no non-root vertex with exactly one child
    		\\ \hline
    	$ \mc T^*_{n, k} $
    		& 	
    													$ n, k \ge 1 $
    		&
    				\niceBullet element of $ \mc T^*_n $
    				\newline
    				\niceBullet path from root to $ \ell_1 $ is length $ k $
    		\\ \hline
    	$ \mc T_n $
    		& 	
    													$ n \ge 2  $ 					
    		&
    				\niceBullet rooted ordered tree with $ n $ leaves 
    				\newline
        			\niceBullet root degree one
        			\newline
        			\niceBullet no non-root vertex with exactly one child
    		\\ \hline
    	$ \mc T_{ n, k } $
    		&											
														$ n, k \ge 2  $
    		&
    				\niceBullet element of $ \mc T_n $
    				\newline
    				\niceBullet path from root to $ \ell_1 $ is length $ k $
    		\\ \hline
     	$ \mc T_{ n, k, l} $
    		&											
														$ n, k \ge 2  $
    		&
    				\niceBullet pair $ ( t, \ell_i ) $
    				\newline
    				\niceBullet $ t \in \mc T_{ n, k } $
    				\newline
    				\niceBullet $ \ell_i $ is a leaf
    				\newline
    				\niceBullet path from $ \ell_i $ to $ \ell_{ i + 1 } $ is length $ l $
    		\\ \hline
    	$ \mc T^0_{n, k, l} $
    		&											
														$ n, k, l \ge 2  $
    		&
    				\niceBullet element of $ \mc T_{ n, k, l } $
    				\newline				
    				\niceBullet $l$-path and leftmost path share no vertices
    		\\ \hline
    	$ \mc T^1_{n, k, l} $
    		&											
														$ n, k \ge 2  $
    		&
    				\niceBullet element of $ \mc T_{ n, k, l } $
    				\newline				
    				\niceBullet $l$-path and leftmost path share exactly 1 vertex
    		\\ \hline
    	$ \mc T^2_{n, k, l} $
    		&											
														$ n, k \ge 2  $
    		&
    				\niceBullet element of $ \mc T_{ n, k, l } $
    				\newline				
    				\niceBullet $l$-path and leftmost path share exactly 2 vertices
    		\\ \hline
    	$ \treeClassFirstLeafShared_{n, k, l} $
    		& 	
    													$ n, k, l \ge 2 $
    		&
    				\niceBullet element of $ \mc T^2_{ n, k, l } $
    				\newline				
    				\niceBullet $l$-path and leftmost path share $ \ell_1 $
    		\\ \hline
    	$ \treeClassRootShared_{n, k, l} $
    		& 	
    													$ n, k, l \ge 2 $
    		&
    				\niceBullet element of $ \mc T^2_{ n, k, l } $
    				\newline				
    				\niceBullet $l$-path and leftmost path share $ \ell_0 $
    		\\ \hline
    	$ \treeClassNeitherShared_{n, k, l} $
    		& 	
    													$ n, k, l \ge 2 $
    		&
    				\niceBullet element of $ \mc T^2_{ n, k, l } $
    				\newline				
    				\niceBullet $l$-path and leftmost path do not share $ \ell_1 $ or $ \ell_0 $
    		\\ \hline
    \end{tabular}

    \captionsetup{ width = 0.9\textwidth }
	\caption{
			Each row defines a type of tree that arises in our analysis.
			}
	\label{table describing trees}
\end{table}

A set of trees will always be denoted by a calligraphic symbol and the non-calligraphic version of that symbol will denote the size of that set.
Each row in our table gives rise to an ordinary generating function whose coefficients are the sizes of the sets appearing in that row, as indicated by the parameter range.
We denote this generating function using a non-calligraphic version of the symbol used for the sets.
For example, the set $ \mc T_{n, k } $ has size $ T_{ n, k } $ and the row containing the sets of the form
$
	\mc T_{n, k}
$
gives rise to the generating function
$
	T(x,y) 
		= 
			\sum_{n,k \ge 2} 
				T_{n,k}
				x^n
				y^k
		.
$
The coefficient of $x^n$ in a generating function $G(x)$ will be denoted by $[x^n]G$.
We also use the natural multivariable extension of this notation.   
For convenience, we extend the definition the coefficients of a generating function to be zero.

\begin{prop}
\label{subbranchRecursion}
The following relation holds:
$$
	T(x, y) = \frac{x y^2 \, T^*(x)}{1 - y T^*(x)}
$$
\end{prop}

\begin{proof}

Fix $ n, k \ge 2 $ and  $ t \in \mc T_{n, k} $.
We obtain $k-1$ disjoint trees by deleting from $ t $
	the root, 
	the first leaf, and 
	the edges of the leftmost path.
We then root each of these trees at the vertex it contains from the leftmost path of $t$ and have them inherit the ordering in $ t $.
In addition, we order these trees into the tuple $(t_1, ..., t_{k-1})$ so that the root of $t_i$ is a vertex that was distance $i$ from the root of $t$. 
These trees, as well as their copies within $ t $, will be referred to as the \textit{spinal subtrees} of $ t $.
See Figure \ref{figSpinalDecomposition}.

\begin{figure}[t]
	\centering
		 \def\svgwidth{0.75\textwidth}
		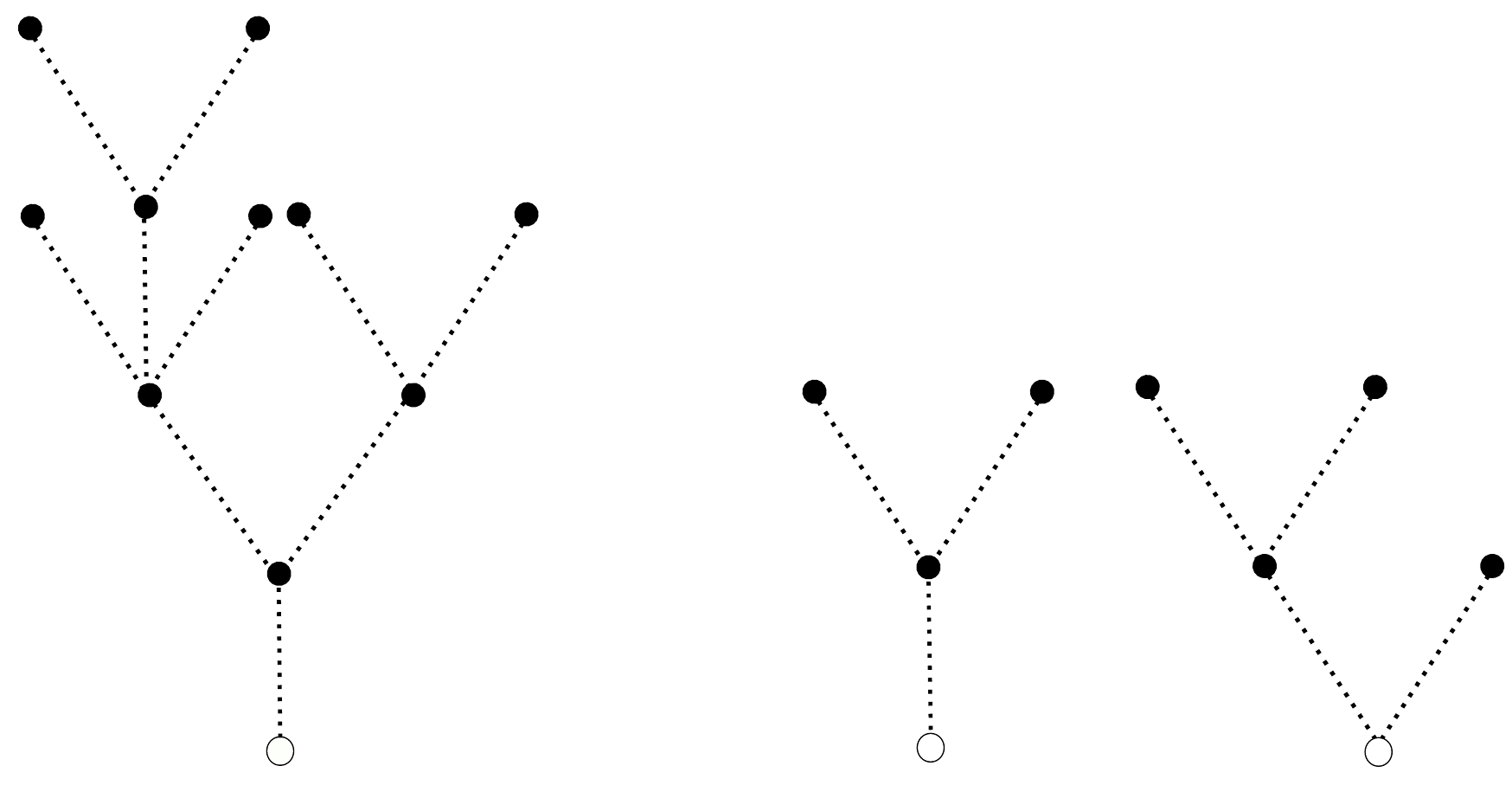
	\caption{Decomposing a tree into its spinal subtrees.}
	\label{figSpinalDecomposition}
\end{figure}

Observe that each $t_i$ lies in some $\mc T^*_{m_i}$ and these values must satisfy $\sum m_i = n-1$. 
On the other hand, all such $ ( k - 1 ) $-tuples of trees are the spinal subtrees of a unique tree in $\mc T_{n, k}$. 
As a result, the decomposition of a tree into its spinal subtrees is a bijection.
In particular, its domain and range have equal size:

\begin{align*}
	T_{n, k} 	& = 	\sum_{m: \sum m_i = n-1} T^*_{m_1} \cdot ... \cdot T^*_{m_{k-1}}, \\
				& = 	[x^{n-1}]\, T^*(x)^{k-1}.
\end{align*}

\noindent
Notice that the above relation holds for $ n, k \ge 2 $. 
Multiplying by $y^k$ and summing over $k$ gives us that 
\begin{align*}
	\sum_{k \ge 2} T_{n, k} y^k 	= 	[x^n] T(x, y)
		& = 	[x^n]\, x \sum _{k\ge 2} T^*(x)^{k-1} y^k \\
		& = 	[x^n]\, xy^2T^*(x) \sum _{k\ge 2} T^*(x)^{k-2} y^{k-2} \\
		& = 	[x^n]\, xy^2T^*(x) (1-yT^*(x))^{-1}
\end{align*}

\noindent 
for all $ n \ge 2 $. As a result, the corresponding generating functions are identical.
\end{proof}

\begin{prop}
\label{addRootRecursion}
The following relation holds:
\begin{align*}
	T^*(x, y) 	& = 	xy + T(x, y) (1 + 1/y)\\
				& = 	x y \frac{1+T^*(x)}{1 - y T^*(x)}
\end{align*}
\end{prop}

\begin{proof}
Fix $n \ge 2$, $ k \ge 1 $, and $ t \in \mc T^*_{n, k} \setminus \mc T_{n, k} $.
Letting $ u $ denote the root of $ t $, we construct a new tree $\hat{t}$ from $ t $ by 
	inserting a new vertex $ v $ as the rightmost child of $ u $,
	and then designating $ v $ as the root of $\hat{t}$ and using the planar order inherited from $t$.

It follows from this construction that the tree $\hat{t}$ lies in $ \mc T_{n, k+1}$ and the map $t \mapsto \hat{t}$ is a bijection onto this set.
As a result, the identity
$$
	T^*_{n, k} 	= 	T_{n, k} + T_{n, k+1}
$$

\noindent 
holds for $n \ge 2$ and $ k \ge 1$. 
When $n = 1$ and $ k \ge 1 $, we have
$$
	T^*_{n, k} 	= 	\mathbbm{1}(k = 1)
$$

\noindent
since $ \mc T^*_1 $ contains only the tree whose leftmost path is exactly one edge.
Setting $T_{1, k} = 0$ for all $k$, these relations can be written concisely as
$$
	T^*_{n, k} 
		= 
			\mathbbm{1}(n = k = 1) + T_{n, k} + T_{n, k+1},
		\qquad
			n, k \ge 1
		.
$$
	
\noindent
The corresponding statement for generating functions,
$$
	T^*(x, y)	=	xy + T(x, y) + T(x, y)/y,
$$
	
\noindent 
is the first of the given identities. 
We obtain the other by applying Proposition \ref{subbranchRecursion}:
\begin{align*}
	T^*(x, y)	
			& = 	
					xy + \frac{x y^2 \, T^*(x)}{1 - y T^*(x)} + \frac{x y \, T^*(x)}{1 - y T^*(x)}
			, 
			\\
			& = 	
					xy \frac{1 + T^*(x)}{1 - y T^*(x)}
			.
\end{align*}
\end{proof}

\begin{prop}
\label{firstLeafSharedIdentity}
The following relation holds:
\begin{align*}
	\genFunFirstLeafShared(x, y, z) 	
			& = 	yz T^*(x, z) (T(x, y) + xy) \\
			& = 	\frac{1+T^*(x)}{1 - z T^*(x)} \frac{x^2 y^2 z^2}{1 - y T^*(x)}.
\end{align*}
\end{prop}

\begin{proof}
Fix a pair $ (t, \ell_1) $ in $ \treeClassFirstLeafShared_{n, k, l}$ and let $t_1$ be the last spinal subtree of $t$. 
Construct a second tree, $t_2$, by deleting the copy of $ t_1 $ in $ t $, and merging the last two edges of the leftmost path. 

Observe that the leftmost path of $t_1$ is essentially the $l$-path of $t$ with one edge removed.
Therefore, $ t_1 $ lies in $ \mc T^*_{m, l-1}$ for some $m$. 
Similarly, the leftmost path of $t_2$ is formed by merging two edges in the leftmost path of $t$, so $ t_2 $ lies in $ \mc T_{n-m, k-1}$ when $k > 2$ and in $ \mc T^*_{1, 1}$ when $k = 2$. 
In either case, the map $ ( t, \ell_1 ) \mapsto (t_1, t_2)$ is a bijection between the relevant sets and we have the relations
\begin{align}
	\genFunFirstLeafShared_{n, 2, l} 	&=		T^*_{n-1, l-1}, 									\label{T^Lidentity1} \\
	\genFunFirstLeafShared_{n, k, l} 	&=		\sum_{m=1}^{n-2} T^*_{m, l-1} T_{n-m, k-1}, 		\label{T^Lidentity2}
\end{align}

\noindent 
which hold for $n, l \ge 2$ and $ k \ge 3$. Introducing a factor of $y^2 z^l$ in (\ref{T^Lidentity1}) and summing over $l$, we have that
\begin{align*}
	\sum_{l \ge 2} \genFunFirstLeafShared_{n, 2, l}	y^2 z^l	& =		y^2 \sum_{l \ge 2} T^*_{n-1, l-1} z^l  \\
		& = 	[x^n]\, x y^2 z T^*(x, z)
\end{align*}

\noindent 
for $n\ge 2$. 
Similarly, from (\ref{T^Lidentity2}) we obtain the identity
\begin{align*}	
	\sum_{\substack{l \ge 2 \\ k \ge 3}} 
		\genFunFirstLeafShared_{n, k, l}	
		y^k z^l
			& =		
					\sum_{m=1}^{n-2}	
						\left(\sum_{l \ge 2} T^*_{m, l-1} z^l \right) 
						\left(\sum_{k \ge 3} T_{n-m, k-1} y^k \right)
			\\
			& = 
					\sum_{m=1}^{n-2} 
						[x^m]\, z \, T^*(x, z) 
						[x^{n-m}]\, y T(x, y) 
			\\
			& = 
					[x^n]\, 
					yz T^*(x, z) T(x, y)
\end{align*}

\noindent
for $n\ge 2$. 
We can combine these into the single relation
$$
	\sum_{\substack{l \ge 2 \\ k \ge 2}} 
		\genFunFirstLeafShared_{n, k, l}	
		y^k z^l	
			=	
				[x^n] 
				yz 
				T^*(x, z) 
				(xy + T(x, y))
			,
$$

\noindent 
which holds for $n\ge 2$. 
Writing this in terms of generating functions, we obtain the first result.
Applying Propositions \ref{subbranchRecursion} and \ref{addRootRecursion} yields the second result:
\begin{align*}
	\genFunFirstLeafShared(x, y, z) 	& = 	y z\,	 \frac{x z (1+T^*(x))}{1 - z T^*(x)} \left(xy + \frac{x y^2 \, T^*(x)}{1 - y T^*(x)}\right), \\
						& = 			 \frac{xyz^2 (1+T^*(x))}{1 - z T^*(x)} \frac{x y}{1 - y T^*(x)} .
\end{align*}
\end{proof}

\begin{prop}
\label{twoSharedRecursion}
The following relation holds:
$$
	T^2(x, y, z) = 2 \genFunFirstLeafShared(x, y, z) + \frac{\genFunFirstLeafShared(x, y, z)^2}{xyz}
$$
\end{prop}

\begin{proof}
Fix a pair $ ( t, \ell ) $  in $\treeClassRootShared_{n, k, l}$ and construct a new pair $ ( t', \ell' ) $ as follows.
The tree $ t' $ is obtained from $ t $ by re-rooting at the leftmost leaf and flipping the resulting tree from left to right. 
The distinguished leaf $ \ell' $ is the first leaf of $ t' $. 
See Figure \ref{figTRtoTLmap}.

\begin{figure}[t]
	\centering
		 \def\svgwidth{0.72\textwidth}
		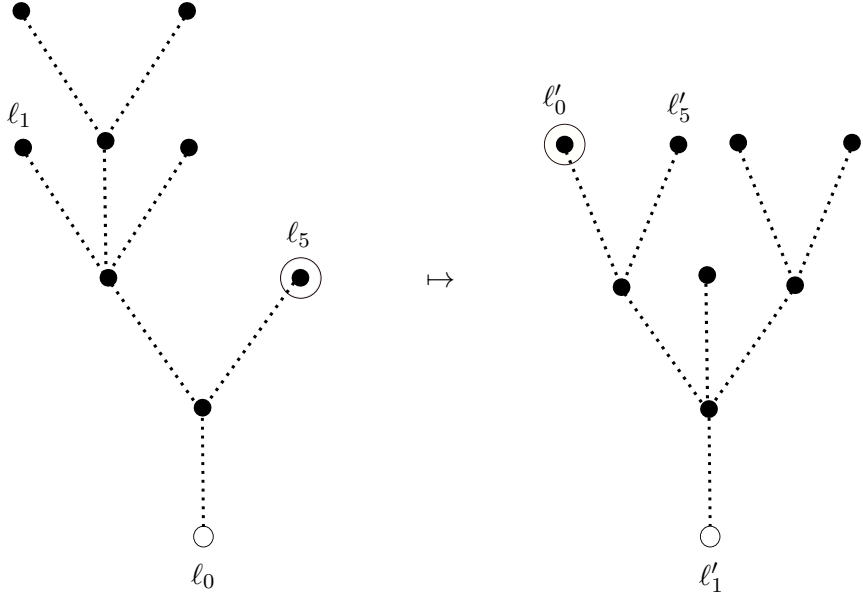
	\caption{The map from $ ( t, \ell ) $ to $( t', \ell' ) $. The image of $\ell_i(t) $ is $\ell_i'$.}
	\label{figTRtoTLmap}
\end{figure}

Observe that the leftmost path and the $l$-path of $ ( t', \ell' ) $ correspond to the leftmost path and the $l$-path of $ ( t, \ell ) $, respectively. 
As a result, we have that $ ( t', \ell' ) \in \treeClassFirstLeafShared_{n, k, l}$. 
In fact, this set is the range of the bijection $ ( t, \ell ) \mapsto ( t', \ell' ) $, so we have the equality
\begin{equation}
	\genFunRootShared_{n, k, l} 
		= 
			\genFunFirstLeafShared_{n, k, l},
		\qquad
			n, k, l \ge 2.
	\label{T^RequalsT^L}
\end{equation}

Let $n, k, l \ge 3$ and fix a pair $ ( t, \ell ) $ in $\treeClassNeitherShared_{n, k, l}$.
Let $u$ and $v$ be the vertices shared by the leftmost path and the $l$-path, with $u$ being closest to the root. 
Next, we delete the edge $ (u, v ) $ so that two new trees form. 
To the tree containing $v$, which we will denote by $t_1$, we add a vertex preceding $v$ and root the tree there. 
To the other tree, which we will denote by $t_2$, we add a vertex descending from $u$ so that it becomes the first leaf in this tree.
We distinguish the last leaf in $ t_1 $ and the first leaf in $ t_2 $ to obtain pairs $ ( t_1, \ell_m ) $ and $ ( t_2, \ell_1 ) $, where $ m $ is the number of leaves in $ t_1 $.

The leftmost paths of $t_1$ and $t_2$ are formed from the leftmost path of $t$ via the removal of an edge and the addition of two edges.
Similarly, the $l$-paths of $ ( t_1, \ell_m ) $ and $ ( t_2, \ell_1 ) $ are formed from the $l$-path of $ ( t, \ell ) $ via the removal of an edge and the addition of two edges. 
Consequently, $ ( t_1, \ell_m ) $ lies in some $ \treeClassRootShared_{m, j, h} $, $ ( t_2, \ell_1 ) $ lies in some $ \treeClassFirstLeafShared_{m', j', h'}$, and these parameters must satisfy 
$ m + m' = n+1, j + j' = k+1 $, and $h + h' = l+1$. 
In fact, the map $ ( t, \ell ) \mapsto ( ( t_1, \ell_m ), ( t_2, \ell_1 ) )$ is a bijection whose range consists of all such pairs. 
This yields the identity
$$
	\genFunNeitherShared_{n, k, l} 	= 	\sum_{m=2}^{n-1} \sum_{j=2}^{k-1}  \sum_{h=2}^{l-1} \genFunRootShared_{m , j, h} \genFunFirstLeafShared_{n+1-m , k+1-j, l+1-h}
$$

\noindent 
for $n, k, l \ge 3$. However, it is easy to verify that equality holds when any of $n, k$, or $l$ is equal to 2 (both sides are zero). 
Combining this with (\ref{T^RequalsT^L}), the above becomes
\begin{align}
\begin{split}
	\genFunNeitherShared_{n, k, l} 	
			& = 	
					\sum_{m=2}^{n-1} 
						\sum_{j=2}^{k-1}  
							\sum_{h=2}^{l-1} 
								\genFunFirstLeafShared_{m , j, h} 
								\genFunFirstLeafShared_{n+1-m , k+1-j, l+1-h} \\
			& = 
					[x^n y^k z^l] \, 
						(xyz)^{-1} 
						\genFunFirstLeafShared(x, y, z)^2.
\end{split}
\label{identityNeitherSharedAndFirstLeaf}
\end{align}

\noindent 
Adding $ \genFunFirstLeafShared_{n, k, l} + \genFunRootShared_{n, k, l} $ and using (\ref{T^RequalsT^L}) once again, we obtain the final relation
\begin{align*}
	T^2_{n, k, l} 	& = [x^n y^k z^l]\,  2 \genFunFirstLeafShared(x, y, z) + [x^n y^k z^l] \, (xyz)^{-1} \genFunFirstLeafShared(x, y, z)^2,
\end{align*}

\noindent 
which holds for all parameter values. 
This establishes the result.
\end{proof}

\bigbreak
\begin{prop}
\label{oneVertexSharedIdentity}
The following relation holds:
$$
	T^1(x, y, z)	=	\frac{\genFunFirstLeafShared(x, y, z)^2}{x y^2 z^2}.
$$
\end{prop}

\begin{proof}
Fix a pair $ ( t, \ell_i ) $ in $\mc T^1_{n,k,l}$ and let $u$ denote the vertex shared between the leftmost and $l$-paths. 
We view the spinal subtree in $ t $ attached to $u$ as being made up of two components: 
	the path from $u$ to $\ell_i$ and the structure to the left of it will be the left component, 
	and the path from $u$ to $\ell_{i+1}$ and the structure to the right of it will be the right component. 
Construct the tree $t'$ from $ t $ by splitting $u$ into two vertices, $u_1, u_2$, so that both lie on the leftmost path, are adjacent, and the spinal subtrees in $ t' $ attached to $u_1, u_2$ are the left and right components respectively. 
Finally, distinguish the $i^{th}$ leaf in $ t' $ to obtain the pair $ ( t', \ell_i ) $.

Since the pair $ ( t', \ell_i ) $
has an $(l+1)$-path that shares two vertices with the leftmost path ($u_1$ and $u_2$), neither of which can be the root or the first leaf, $ ( t', \ell_i ) $ lies in $\treeClassNeitherShared_{n, k+1, l+1}$. 
In fact, this set is the range of the bijective map $ ( t, \ell_i ) \mapsto ( t', \ell_i ) $, so the relation
$$
	T^1_{n,k,l} 	= 	\genFunNeitherShared_{n, k+1, l+1}
$$
	
\noindent 
holds for all values of $n, k$, and $l$. 
Using (\ref{identityNeitherSharedAndFirstLeaf}), we can rewrite this as 
\begin{align*}
	T^1_{n,k,l}		&=		[x^n y^{k+1} z^{l+1}]\,  (xyz)^{-1} \, \genFunFirstLeafShared(x, y, z)^2   \\
					&=		[x^n y^k z^l]\,  x^{-1}(yz)^{-2} \, \genFunFirstLeafShared(x, y, z)^2,
\end{align*}

\noindent from which we have the result.
\end{proof}


\begin{prop}
\label{noneSharedIdentity}

The following relation holds:
$$
	T^0(x, y, z)	
			=	
				T^1(x, y, z)
				\deriv{x}T^*(x).
$$
\end{prop}

\begin{proof}
Given a pair $ ( t, \ell ) $ in $ \mc T^0_{n, k, l}$, let $ \bar{t} $ be the spinal subtree in $ t $ containing the $ l $-path.
Denote by $ u $ the vertex on the $ l $-path that is closest to the `root' of $ \bar{t} $.
Form the tree $ t_1 $ by taking a copy of $ t $ and replacing its copy of $ \bar{t} $ with the subtree formed by $ u $ and its descendants.
Form the tree $ t_2 $ by taking a copy of $ \bar{t} $ and turning its copy of $ u $, which we denote by $ u^* $, into a leaf by deleting its descendents.
Distinguish the leaf in $ t_1 $ that is a copy of $ \ell $ to obtain the pair $ ( t_1, \ell' ) $ and distinguish the newly created leaf $ u^* $ in $ t_2 $ to obtain the pair $ ( t_2, u^* ) $.
See Figure \ref{figDisjointPathMap}. 

\begin{figure}[t]
    \centering

	\def\svgwidth{0.8\textwidth} 
	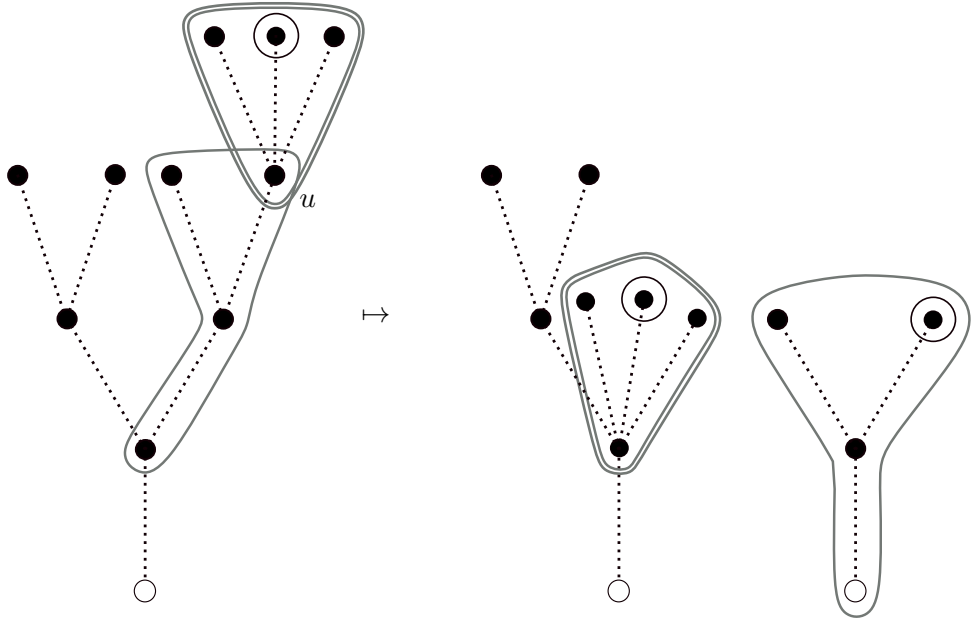

    \captionsetup{ width = 0.85\textwidth }
	\caption{
				The decomposition of $ ( t, \ell ) $ into $ ( ( t_1, \ell' ), ( t_2, u^* ) ) $. 
				Root vertices are white, distinguished leaves are circled, and circled subtrees on the left map to the corresponding circled subtree on the right.
			}
	\label{figDisjointPathMap}
\end{figure}

The pair $ ( t_1, \ell' ) $ must lie in some $\mc T^1_{m_1, k, l}$ since its leftmost path and $l$-path are copies of the ones in $ ( t, \ell ) $ except that the $l$-path of $ t_1 $ now shares one vertex with the leftmost path (some copy of $ u $).
The tree $ t_2 $ must lie in some $ \mc T^*_{m_2} $ since it is derived from a spinal subtree.
This time, the leaf parameters satisfy $ m_1 + m_2 = n + 1 $ since all of the leaves in $ t_1 $ and $ t_2 $ are derived from leaves in $ t $ with the exception of the newly created leaf $ u $.
The bijection $ ( t, \ell ) \mapsto ( ( t_1, \ell' ), ( t_2, u ) ) $ yields the relation
$$
	T^0_{n, k, l} 	
		= 	
			\sum_{ m = 1 }^{ n - 1 }
				T^1_{n + 1 - m, k, l} 
				\,
				m \, T^*_m 
		,
		\qquad  
			n, k, l \ge 2.
$$

\noindent
The entire parameter range is included here because both sides are zero when the bijection is not well-defined.
Introducing a factor of $ y^k z^l $ and summing over $ k $ and $ l $, we obtain
\begin{align*}
	\sum_{ k, l \ge 2 }
		T^0_{n, k, l}
		y^k z^l
			& =		
					\sum_{ m = 1 }^{ n - 1 }
						\left(\sum_{k, l \ge 2} T^1_{n + 1 - m, k, l} y^k z^l \right) 
						\,
						m \, T^*_m 
			\\
			& =		
					\sum_{ m = 1 }^{ n + 1 }
						 [x^{n + 1 - m}] 
						 	T^1(x, y, z)
						\,
						[x^m]
							x 
							\deriv{x} T^*(x) 
			\\
		 	& = 	
					[x^{n + 1}] 
						T^1(x, y, z)
						x \deriv{x} T^*(x) 
			\\
		 	& = 	
					[x^n] 
						T^1(x, y, z)
						\deriv{x} T^*(x) 
\end{align*}

\noindent 
for all $ n $, which corresponds to the given generating function relation.
\end{proof}

\section{Asymptotic Expansions} 
In this section, we obtain estimates on our generating functions' coefficients that lead to the estimates on the moments of the $ \zeta_k $ in Lemma \ref{momentEstimates}.
Our approach will use the multivariate analytic combinatorics tools developed in \cite{gao00}.

To begin, we recall some definitions from \cite{gao00}. 

\begin{defn}
For every $\eps > 0 $ and $\phi \in ( 0, \frac{\pi}{2} )$, we associate the \textit{$\delta$-domain} 
$$
	\deltaDomain
		=
			\{
				x \in \mbb{C} 
			: 
				|x| \le 1 + \eps,
				\, 
				x \neq 1, 
				\,
				|\text{Arg}(x-1)| \ge \phi
			\}.
$$
For every $\eps > 0 $, $\phi \in ( 0, \frac{\pi}{2} )$, and positive integer $m$, we associate the \textit{$\mc R$-domain} 
$$
	\rmDomain 
		=
			\{ 
				(x, y_1, ..., y_m) \in \mbb{C}^{m+1} 
			: 
				x \in \deltaDomain, 
				|y_j| < 1 \text{ for all } j
			\}.
$$

\end{defn}

\noindent
Whenever the parameter $ m $ can be deduced from context, we will omit it in $ \rmDomain $ and use $ \mb y $ as a shorthand for the list of variables $ y_1, ... , y_m $.

\begin{defn}
We write 
$$
	f( x, y_1, ... , y_m ) = \widetilde{O}\Big((1-x)^{-\alpha}\prod_{j=1}^m (1-y_j)^{-\beta_j}\Big)
$$

\noindent 
if there exist 
	some $ \rDomain $,
	a real number $\alpha'$,
	$\beta' \ge 0$, and
	$\gamma \ge 0$ 
such that
\begin{enumerate}[label = (\roman*)]
	\item 
	$ f $ is analytic on $ \rDomain $,
	
	\item 
	$
		f(x, \mb y ) 
			= 
				O \Big( (1-x)^{-\alpha'}\prod_{j=1}^m (1-|y_j|)^{-\beta'} \Big)
	$
	on $ \rDomain $, and

	\item 
	$
		f(x, \mb y ) = O \Big((1-x)^{-\alpha}\prod_{j=1}^m (1-|y_j|)^{-\beta_j}\Big)
	$
	as $(1-x)(1-y_i)^{-\gamma} \to 0$ for all $i$ and $ ( x, \mb y ) \in \rDomain $.

\end{enumerate}

\end{defn}

\begin{defn}
We write 
$$ 
	f ( x, y_1, ... , y_m ) 
		\approx 
			c (1-x)^{-\alpha}\prod_{j=1}^m (1-y_j)^{-\beta_j}
$$
if $ c \neq 0 $ and there exist 
	$ \alpha' < \alpha $,
non-negative numbers $ \beta_1', ..., \beta_m' $,
and functions 
	$ C( \mb y ) $, 
	$ C_0( x, \mb y ) $, ..., $ C_m( x, \mb y ) $,
	and
	$ E( x, \mb y ) $
such that
\begin{enumerate}[label=(\roman*)]

	\item
	$
		f ( x, \mb y ) 
			= 
					C( \mb y ) (1-x)^{-\alpha}\prod_{j=1}^m (1-y_j)^{-\beta_j}
				+	\sum_{j=0}^m
						C_j( x, \mb y )
				+
					E( x, \mb y )
			,
	$

	\item
	$ C $ is analytic and
	$ 
		C( \mb y ) 
			= 
				c 
			+ 	O\left( \sum_{j = 1}^m | 1 - y_j | \right)
	$
	on some product of $\delta$-domains,
	
	\item
	$ C_0, C_1, \ldots, C_m $ are polynomials in $ x, y_1, \ldots, y_m $ respectively, and
	
	\item
	$
		E( x, \mb y ) 
			= 
				\widetilde{O}\Big( (1-x)^{-\alpha'} \prod_{j=1}^m (1-y_j)^{-\beta_j'} \Big)
			.
	$

\end{enumerate}

\end{defn}

Next, we state the tools in \cite{gao00} that we will need.  Our statements differ slightly from those in \cite{gao00}, but only in that for our purposes we found it helpful make explicit the constants implied by their big-oh notation.

\begin{prop}[Lemma 2, \cite{gao00}]
	\label{prop GW Lemma 2}
	Suppose that 
	$$
		f( x, y_1, ... , y_m ) 
			= 
				\widetilde{O}
					\Big(
						(1-x)^{ -\alpha }
						\prod_{j=1}^m 
							(1-y_j)^{ -\beta_j }
					\Big)
    		.
	$$

	\noindent
	Then the following statements hold:
	\begin{enumerate}[ label = (\roman*) ]
		
		\item
		for every $ L > 0 $ there exists some $ M > 0 $ such that
    	$$
    		\Big|
    			[ x^n 
				y_1^{k_1} 
				\ldots 
				y_m^{ k_m } 
				]
            	f( x, \mb y )
    		\Big|	
    			\le
    				M
					n^{ \alpha - 1 }
					\prod_{ j = 1 }^m
						k_j^{ \beta_j }
    	$$
    
    	\noindent
    	whenever $ 1 \le k_j \le L \log n $ for all $ j $ and $ n $ is large, and
		
		\item
		for every $ \eps \in ( 0, 1 ) $ there exists some $ M > 0 $ such that
    	$$
    		\Big|
    			[ x^n 
				y_1^{k_1} 
				\ldots 
				y_m^{ k_m } 
				]
            	f( x, \mb y )
    		\Big|	
    			\le
    				M
					n^{ \alpha - 1 }
					\prod_{ j = 1 }^m
						\eps^{ - k_j }
    	$$
    
    	\noindent
    	for all $ k_1, \ldots, k_m, $ and $ n $.
		
	\end{enumerate}
\end{prop}

\begin{prop}[Lemma 3, \cite{gao00}]
	\label{prop GW Lemma 3}	
	Suppose that 
    $$ 
    	f ( x, y_1, ... , y_m ) 
    		\approx 
    			c (1-x)^{-\alpha}
				\prod_{j=1}^m 
					( 1 - y_j )^{ -\beta_j }
    $$

	\noindent
	and $ \alpha \neq 0, -1, -2, \ldots $.
	Then for every $ L > 0 $ there exists some $ M > 0 $ such that
	\begin{align*}
		\left|
			[ x^n 
				y_1^{k_1} 
				\ldots 
				y_m^{ k_m } 
				]
        	f( x, \mb y )
		-
			c\,
			\frac{ 
					n^{ \alpha - 1 }
				}{
					\Gamma( \alpha )
				}
			\prod_{ j = 1 }^m
			\frac{ 
					k_j^{ \beta_j - 1 }
				}{
					\Gamma( \beta_j )
				}
		\right|	
			\le
				M
				n^{ \alpha - 1 }
    			\left(
    				\prod_{ j = 1 }^m
        				k_j^{ \beta_j - 1 }
				\right)
    			\sum_{ j = 1 }^m
        			\frac{1}{k_j}
	\end{align*}

	\noindent
	whenever $ k_j \le L \log n $ for all $ j $ and each $ k_j $ is large.

\end{prop}

Our analysis begins with the function $T^*(x)$. 
For convenience, we work with the rescaled version
\begin{equation}
\label{scaledTstarDef}
		\scaledTstar(x) 	
				= 		
					2rs \, T^*(x/r), 
\end{equation}

\noindent 
where 
$ 
	r = 3 + 2 \sqrt{2} 
$ 
and 
$ 
	s = ( \sqrt{2} - 1 )/2.
$
It will be useful to note the identities
\begin{equation}
	s ( r - 1 ) 	= 	1
					= 	2 s \sqrt{r}
		.
\label{rsIdentities}
\end{equation}
\noindent
The following result summarizes the properties of $\scaledTstar(x)$.

\begin{prop}
\label{scaledTstarProperties}

The following statements hold: 
\begin{enumerate}[label=(\roman*)]
	\item
	$\scaledTstar(x)$ has non-negative coefficients, $\scaledTstar_n$, and there exists a nonempty set of indices, $J$, so that $\scaledTstar_j > 0$ for $j \in J$ and 
	$ \gcd\{i-j: i, j \in J\} = 1$,

	\item
	$\scaledTstar(x)$ can be regarded as a function on $ \mbb C $ that is analytic on every $\delta$-domain,

	\item
	$ \scaledTstar(x) $ has the form 
	$$
		\scaledTstar(x) 
			= 
				1 - qs(1-x)^{1/2} + s(1-x) - s(1-x)^{1/2} \sigma(x),
	$$ 
	
	\noindent
	where $q = (r^2-1)^{1/2}$, $\sigma(x) = (r^2-x)^{1/2} - q$,
	and $ \sigma(x) ( 1 -x  )^{-1} $ is bounded on every $\delta$-domain, 
	
	\item
	$|\scaledTstar(x)| \le 1$ on some $\delta(\eps, \phi)$, and

	\item
	$ \scaledTstar'(x) $ has the form
	$$
		\scaledTstar'(x)
			=  
				\frac{qs}{2}(1-x)^{-1/2} + u(x),
	$$ 
	
	\noindent
	where $ u(x) $ is analytic and bounded on every $\delta$-domain.
	
\end{enumerate} 

\end{prop}

For the remainder of the paper we will fix $\eps_0$ and $\phi_0$ such that $|\tau(x)|\leq 1$ on  $\delta(\eps_0, \phi_0)$.

\begin{proof}

To see that (i) holds, observe that $\scaledTstar_0 = 0$ and all other coefficients are positive. Thus, $J$ can be taken to be $\mbb{N}$.

Setting $y = 1$ in the identity in Proposition \ref{addRootRecursion}, we find that $T^*(x)$ must satisfy
$$
	T^*(x) 
		= 
			x \, 
			\frac{ 1 + T^*(x) }{1 - T^*(x)}
	.
$$

\noindent 
Together with the condition $ T^*_0 = 0 $, this equation leads to the explicit form
\begin{align}
	\begin{split}
		T^*(x)	
			& = 	
					\frac{1 - x - \sqrt{x^2-6x+1}}{2}
			\\
			& = 	\frac{1 - x - \sqrt{(1/r-x)(r-x)}}{2}
			.
	\end{split} 
		\label{T*rootform}
\end{align}

\noindent 
This representation allows us to regard $T^*(x)$ as a function on $ \mbb C $ that is analytic on $\mbb C \setminus [\frac{1}{r}, r]$. 
Consequently, $\scaledTstar(x)$ is analytic on $\mbb C \setminus [1, r^2]$, and (ii) holds.

The expansion in (iii) follows from (\ref{T*rootform}):
\begin{align*}
	\scaledTstar(x) 	& = 	rs (1 - x/r - (1/r - x/r)^{1/2}(r - x/r)^{1/2}) \\
						& = 	s (r - 1 + 1 - x - (1-x)^{1/2}(q + \sigma(x)), \\
						& = 	1 + s(1-x) - qs(1-x)^{1/2} - s(1-x)^{1/2} \sigma(x).
\end{align*}

\noindent
To obtain the bound on  $ (1 - x)^{-1} \sigma(x) $, we first observe that $ \sigma(1) = 0 $ and $ \sigma(x) $ is differentiable at $ x = 1 $.
As a result, the quantity $ (1 - x)^{-1} \sigma(x) $ is a difference quotient, whereby it must be bounded near $ x = 1 $.
Away from $ x = 1 $, the quantity $ ( 1 - x )^{-1} $ is clearly bounded, and within any $ \delta $-domain, $ \sigma( x ) $ is bounded. 

The bound in (iv) follows directly from Lemma 4 in \cite{gao00}.
The boundedness claimed in (v) follows from the explicit form
\begin{align*}
	u(x)
			& = 	
					\scaledTstar'(x) 
				-	\frac{qs}{2}(1-x)^{-1/2}
			\\
			& = 	
				-	s
				-	s
					(1-x)^{1/2} 
					\sigma'(x)
				+	\frac{s}{2}
					(1-x)^{-1/2} 
					\sigma(x)
			\\
			& = 	
				-	s
				+	\frac{s}{2}
					(1-x)^{1/2} 
					(r^2-x)^{-1/2}
				+	\frac{s}{2}
					(1-x)^{-1/2} 
					\sigma(x)
\end{align*}
and the boundedness established in (iii).
The analyticity follows from the analyticity established in (ii), which $ \tau'(x) $ inherits from $ \tau(x) $.
\end{proof}

The above result immediately gives us our first coefficient estimate.

\begin{prop}
\label{propUnivariateCoefficients}
The coefficients of $T(x)$ are given by
$$
	T_n = \frac{qs^2}{2 \sqrt{\pi}}r^n n^{-3/2}(1 + O(1/n))
$$

\noindent as $n \to \infty$.
\end{prop}

This result is well known, having first been obtained in \cite{flajolet1999analytic}, but we give a proof to illustrate our approach.

\begin{proof}
Using Proposition \ref{addRootRecursion} and (\ref{scaledTstarDef}), we find that $T(x)$ and $\scaledTstar(x)$ are related by the following equation:
$$
	T(x/r) 	= 	s \scaledTstar(x) - 2s^2 x.
$$

\noindent 
Together with Proposition \ref{scaledTstarProperties}, this tells us that $ T(x/r) $ can be regarded as an analytic function and has the expansion 
$$
	T(x/r) 	= 	s - qs^2(1-x)^{1/2} + s^2 (1-x) + O(1-x)^{3/2} - 2s^2 x
$$

\noindent 
on some $\delta$-domain. 
It is known (see Theorems VI.1-3 in \cite{flajolet09})
that a function of this form has coefficients given by
\begin{align*}
	r^{-n} T_n 	& = 	[x^n] T(x/r) \\
				& = 	- qs^2 [x^n] (1-x)^{1/2} + O([x^n](1-x)^{3/2}) \\
				& = 	\frac{qs^2}{2 \sqrt{\pi}} n^{-3/2}(1 + O(1/n)) + O(n^{-5/2}) \\
				& = 	\frac{qs^2}{2 \sqrt{\pi}} n^{-3/2}(1 + O(1/n))
\end{align*}

\noindent
as $ n \to \infty $.
\end{proof}

As in the above proof, our first step to obtaining a coefficient estimate is to identify a generating function as an analytic function.
In our next result, we complete this step for all of our generating functions.

\begin{prop}
	\label{analyticityOfGeneratingFunctions}
	Under some rescaling, each generating function in Section \ref{sectionFunctionalEquations} can be regarded as an analytic function on some $ \mc R $-domain.
	In particular,
	\begin{enumerate}[label=(\roman*)]
	
		\item
		$ T( x/r, y/2s ) $ and $ T^*( x/r, y/2s ) $ are analytic functions on $ \mc R^1( \eps_0, \phi_0 ) $, and
	
		\item
		$ \genFunFirstLeafShared( x/r, y/2s, z/2s ) $, 
		$ T^2( x/r, y/2s, z/2s ) $, 
		$ T^1( x/r, y/2s, z/2s ) $, 
		and 
		$ T^0( x/r, y/2s, z/2s ) $ 
		are analytic functions on $ \mc R^2( \eps_0, \phi_0 ) $.
	\end{enumerate}

\end{prop}

\begin{proof}
	We first show that the formal power series
	$$
	f( x, y ) 
				= 
				 		\left( 1 - \frac{y}{2s} T^*( x/r ) \right)^{ -1 }
				=
						\sum_{ k \ge 0 }
							\left( \frac{y}{2s} T^*( x/r ) \right)^k
	$$
	
	\noindent
	can be identified as an analytic function on $ \mc R^1( \eps_0, \phi_0 ) $. 
	Since $ T^*(x/r) $ has been identified as an analytic function on every $ \delta $-domain (see Proposition \ref{scaledTstarProperties}(ii)), it suffices to establish the bound
	$$
		\left|
				\frac{y}{2s}
				T^*( x/r )
		\right|
				<
						1
	$$	
	
	\noindent

	on $ \mc R^1( \eps_0, \phi_0 ) $.
	This bound follows from the second identity in (\ref{rsIdentities}) and the bound in Proposition \ref{scaledTstarProperties}(iv).
	
	Given the analyticity of $ T^*(x/r) $ and $ f( x, y) $, Propositions \ref{subbranchRecursion}, \ref{addRootRecursion}, and \ref{firstLeafSharedIdentity} establish (i) and the analyticity of $ \genFunFirstLeafShared(x/r, y/2s, z/2s ) $.
	Applying Propositions \ref{firstLeafSharedIdentity}, \ref{twoSharedRecursion}, \ref{oneVertexSharedIdentity}, and \ref{noneSharedIdentity} concludes the proof.
\end{proof}


The next step in the procedure is to obtain singular expansions for our generating functions.
It will be helpful to first analyze the functions
$$
	A_k(x, y) 	
		= 	
			\frac{x^k}{1 - y \scaledTstar(x)}
		\qquad
		\text{and}
		\qquad
	B_k(x, y) 
		= 	
			A_k(x, y) \, \scaledTstar(x)
		,
$$
	
\noindent 
defined for each non-negative integer $k$ on $ \myRdomain $. 
The basic properties of these functions are summarized in the following proposition.

\begin{prop}
	\label{generalExpansions}
	On $ \myRdomain $, each $ A_k $ and $ B_k $
	\begin{enumerate}[label=(\roman*)]
	
		\item 
		is analytic,
		
		\item 
		is bounded by a multiple of $D(y) = (1-|y|)^{-1}$, and
	
		\item 
		has a singular expansion given by
		\begin{equation}
			A_k(x, y)	
				= 	
					\frac{1}{1 - y} - \frac{ qs y (1-x)^{1/2} }{(1 - y)^2} + \widetilde{O}( (1-x)(1 - y)^{-3} )
				, 
		\label{expansionA} 
		\end{equation}
		
		\noindent
		or
		\begin{equation}
			B_k(x, y)	
				= 	
					\frac{1}{1 - y} - \frac{ qs (1-x)^{1/2} }{(1 - y)^2} + \widetilde{O}( (1-x)(1 - y)^{-3} ) 
				.
		\label{expansionB}
		\end{equation}
	\end{enumerate}
\end{prop}

\begin{proof}
The analyticity in (i) can be established as in Proposition \ref{analyticityOfGeneratingFunctions}.
The bound in (ii) follows from Proposition \ref{scaledTstarProperties}(iv).
To verify the expansion for $A_0(x, y)$, we define
$$
	E(x, y)	
		=
			A_0(x, y) - \frac{1}{1 - y} + \frac{qs y (1-x)^{1/2}}{(1 - y)^2}.
$$

\noindent 
It follows from (i) that $E(x, y)$ is analytic on $\myRdomain$. 
To obtain a bound on $E(x, y)$, we write
\begin{align*}
	E(x, y) 
	(1-y)
	(1-y\scaledTstar(x))	
			& = 	
					(1-y) - (1-y\scaledTstar(x)) + qs y (1-x)^{1/2} \frac{1-y\scaledTstar(x)}{1 - y} \\
			& = 	
					y(\scaledTstar(x)-1) + qs y (1-x)^{1/2} \bigg(1 + y \, \frac{1-\scaledTstar(x)}{1-y}\bigg) \\
			& = 	
					sy (1-x) - sy (1-x)^{1/2} \sigma(x) \\
			& \quad
					+ qs y^2 (1-x)^{1/2} \, \frac{1-\scaledTstar(x)}{1-y}
			,
\end{align*}

\noindent 
and
observe that both $\sigma(x)$ and $ 1 - \scaledTstar(x)$ can be bounded by a constant times $|1-x|^{1/2}$ on $\delta(\eps_0, \phi_0)$.
From this,
it follows that
$$
	|E(x, y)| 	\le 	M \frac{|1-x|}{(1 - |y|)^3}
$$

\noindent 
on $\myRdomain$ for some $M > 0$,
establishing
the expansion for $A_0(x, y)$.

To verify that $A_k(x, y)$ and $A_0(x, y)$ have the same expansion, we simply note that their difference is negligible:
\begin{align*}
	A_k(x, y) - A_0(x, y)	& = 	\frac{x^k-1}{1-y\scaledTstar(x)} \\
							& = 	\widetilde{O} \left( (1-x)(1-y)^{-1} \right).
\end{align*}

\noindent 
Similarly, the relationship given by
\begin{align*}
	A_k(x, y) - B_k(x, y)	
			& = 	
					A_k(x, y) ( 1 - \scaledTstar(x) ) 
			\\
			& = 	
					\left(
							(1-y)^{-1} + \widetilde{O}\left( (1-x)^{1/2}(1-y)^{-3} \right)
					\right) 
			\\
			& \quad 
					\times \big( qs(1-x)^{1/2} + O(1-x) \big) 
			\\
			& = 	qs(1-x)^{1/2}(1-y)^{-1} + \widetilde{O}\left( (1-x)(1-y)^{-3} \right)
\end{align*}

\noindent 
verifies that the expansions for the $B_k(x, y)$ hold.
\end{proof}

The above result gives us our next singular expansion, and with it, our next coefficient estimates.

\begin{prop}
\label{propBivariateCoefficients}
The following statements hold:
\begin{enumerate}[label = (\roman*)]
		\item
		\label{claim bivariate expansion}
		$T(x/r, y/2s) \approx -2qs^2 (1-x)^{1/2} (1-y)^{-2}$,
		
		\item
		\label{claim bivariate coefficient estimate all k}
		for every $ p \in ( 0, r^{ 1/2 } ) $ there exists some $ M > 0 $ such that
    	$$
			T_{n, k}
    			\le
    				M
					r^n
					n^{ -3/2 }
					p^{ -k }
    	$$

		\noindent
		for all $ n, k \ge 2 $, and

		\item
		\label{claim bivariate coefficient estimate log k}
		for every $ L > 0 $ there exists some $ M > 0 $ such that
    	$$
    		\left|
    			T_{n, k}
    		-
    			\frac{ 
            			qs^2
    				}{
    					\sqrt{\pi}
    				}
				\,
    			r^{n-k/2}
    			n^{ -3/2 }
    			k
    		\right|	
    			\le
    				M
    				r^{n-k/2}
					n^{ -3/2 }
    	$$

    	\noindent
    	whenever $ k \le L \log n $ and $ k $ is large.
		
\end{enumerate}
\end{prop}

\begin{proof}

	Combining Propositions \ref{subbranchRecursion} and \ref{generalExpansions}, we obtain the expansion
	\begin{align*}
		T(x/r, y/2s)	
			& =		
					\frac{2s x y^2 \scaledTstar(x)}{1 - y \scaledTstar(x)} 
			\\
			& =		
					2sy^2 B_1(x, y)
			\\
			& = 		
					\frac{2s y^2 }{1 - y} - \frac{  2qs^2  y^2 (1-x)^{1/2}}{(1 - y)^2} + \widetilde{O}\left((1-x)(1-y)^{-3} \right)
			.
	\end{align*}
	
	\noindent 
	Letting 
	$
		C_0(x, y) = 2s y^2 (1 - y)^{-1}
	$ 
	and 
	$ 
		C(y) = - 2 q s^2 y^2 
	$, 
	we have that 
		$ C_0( x, y ) $ is a polynomial in $x$, 
		$ C(1) = -2 q s^2 \neq 0 $, 
		$ C $ is analytic on every $\delta$-domain, and 
		$ C(y) = C(1) + O(1-y) $ on every $\delta$-domain. 
	This establishes (i).

	Let $ p \in ( 0, r^{ 1/2 } ) $.
	The above expansion and Proposition \ref{prop GW Lemma 2} give us that
	\begin{align*}
		[x^n y^k ]
		T(x/r, y/2s)	
			& =
					[x^n y^k ]
					\frac{2s y^2 }{1 - y} 
				+
					[x^n y^k ] 
					\widetilde{O}(1-x)^{1/2}(1-y)^{-3}
			\\
			& \le
    				M
					n^{ -3/2 }
					\left( \frac{ p }{ r^{1/2} } \right)^{ -k }
	\end{align*}

	\noindent
	for some $ M > 0 $ and all $ n, k \ge 2 $.
	Since
	$
		T_{n, k} 	
				= 
					r^n (2s)^k \,
					[x^n y^k] 
					T(x/r, y/2s) 
				,
	$
	we can introduce the factor $ r^n (2s)^k = r^{ n - k/2} $ to obtain
	$$
		T_{n, k}	
			\le
    				M
					r^n 
					n^{ -3/2 }
					p^{ -k }
			,
	$$

	\noindent
	establishing \ref{claim bivariate coefficient estimate all k}.

 Let $ L > 0 $. 
 Combining \ref{claim bivariate expansion} and Proposition \ref{prop GW Lemma 3}, we obtain the bound
	\begin{align*}
		\left|
				[x^n y^k] 
				T(x/r, y/2s) 
		-
			( -2qs^2 )\,
			\frac{ 
					n^{ -3/2 }
				}{
					\Gamma( - \frac{1}{2} )
				}
			\frac{ 
					k
				}{
					\Gamma( 2 )
				}
		\right|	
			\le
				M
				n^{ -3/2 }
	\end{align*}

	\noindent
	for some $ M > 0 $ whenever $ k \le L \log n $ and $ k $ is large.
	Multiplying by $ r^n (2s)^k = r^{ n - k/2} $, this becomes
	\begin{align*}
		\left|
				T_{n, k}
		-
			\frac{ 
					-2 q s^2
				}{
					-2 \sqrt{ \pi }
				}
				\,
				r^{ n - k/2}
				n^{ -3/2 }
				k
		\right|	
			\le
				M
				r^{ n - k/2}
				n^{ -3/2 }
			,
	\end{align*}

	\noindent
	establishing \ref{claim bivariate coefficient estimate log k}.

\end{proof}


It only remains to obtain a singular expansion for $ T( x/r, y/2s, z/2s) $.
This is done in the next two results.

\begin{prop} 
\label{trivariateExpansions}

The following statements hold:
\begin{enumerate}[label = (\roman*)]
		\item
		\label{claim T1 expansion}
		$
			T^1(x/r, y/2s, z/2s) 	
				= 	
					\widetilde{O}\left( (1-y)^{-6} (1-z)^{-6} \right)
				,
		$

		\item
		\label{claim T2 expansion}
		$
			T^2(x/r, y/2s, z/2s)
				= 	
					\widetilde{O}\left( (1-y)^{-6} (1-z)^{-6} \right)
				,
		$
		and
		
		\item
		\label{claim T0 expansion}
		$
			\displaystyle
			T^0(x/r, y/2s, z/2s) 	
				\approx
					2 q s^2
					(1-x)^{-1/2}
					(1 - y)^{-2} 
					(1 - z)^{-2}
				.
		$
\end{enumerate}

\end{prop}

\begin{proof}

Applying Proposition \ref{generalExpansions} leads to the following expansion:
\begin{align*}
	&
		A_0( x, y ) 
		( 
			A_1( x, z )
		+	2 s 
			B_1( x, z )
		)
	\\
		& \hspace{ 12 mm } = 	
			\left(
					\frac{1}{1 - y}
				 + 	\widetilde{O}\left( (1-x)^{1/2} (1 - y)^{-3} \right)
			\right)
			\left(
					\frac{ 1 + 2 s }{1 - z}
				 + 	\widetilde{O}\left( (1-x)^{1/2} (1 - z)^{-3} \right)
			\right)
		\\
		& \hspace{ 12 mm } = 	
					\frac{ 1 + 2 s }{(1 - y)(1 - z)}
				 + 	\widetilde{O}\left( (1-x)^{1/2} (1 - y)^{-3} (1 - z)^{-3} \right)
		.
\end{align*}

\noindent
Combining this with the identity $ 1 + 2 s = q^2 s^2 $ and Proposition \ref{firstLeafSharedIdentity}, we obtain
\begin{align*}
	&
	\genFunFirstLeafShared(x/r, y/2s, z/2s)
	\\
		& \hspace{ 15 mm } = 	
			x y^2 z^2
			( 1 + 2s \scaledTstar(x) )	 
			\frac{1}{1 - y \scaledTstar(x)} \frac{x}{1 - z \scaledTstar(x)} 
		\\
		& \hspace{ 15 mm } = 	
			x y^2 z^2 
			( 1 + 2s \scaledTstar(x) )
			A_0( x, y ) A_1( x, z )
		\\
		& \hspace{ 15 mm } = 	
			x y^2 z^2
			A_0( x, y ) 
			( A_1( x, z ) + 2s B_1( x, z ) )
		\\
		& \hspace{ 15 mm } = 	
			x y^2 z^2
			\left(
					\frac{
							q^2 s^2
						}{
							(1 - y)
							(1 - z)
						}
				 + 	\widetilde{O}\left( (1-x)^{1/2} (1 - y)^{-3} (1 - z)^{-3} \right)
			\right)
		.
\end{align*}

\noindent 
Substituting this into the identity of Proposition \ref{oneVertexSharedIdentity}, we obtain \ref{claim T1 expansion}:
\begin{align}
\begin{split}
	&
		T^1(x/r, y/2s, z/2s) 
	\\
			& \hspace{ 10 mm } = 	
						4 s^2 x^{-1} y^{-2} z^{-2} 
						\genFunFirstLeafShared(x/r, y/2s, z/2s)^2
			\\
			& \hspace{ 10 mm } = 	
						4 s^2 x y^2 z^2
				\left(
						\frac{
								q^2 s^2
							}{
								(1 - y)
								(1 - z)
							}
					 + 	\widetilde{O}\left( (1-x)^{1/2} (1 - y)^{-3} (1 - z)^{-3} \right)
				\right)^2
			\\
			& \hspace{ 10 mm } = 	
						4 s^2 x y^2 z^2
				\left(
						\frac{
								q^4 s^4
							}{
								(1 - y)^2
								(1 - z)^2
							}
					 + 	\widetilde{O}\left( (1-x)^{1/2} (1 - y)^{-6} (1 - z)^{-6} \right)
				\right)
			\\
			& \hspace{ 10 mm } = 	
						( 1 - ( 1 - x ) )
				\left(
						\frac{
								4 q^4 s^6 y^2 z^2
							}{
								(1 - y)^2
								(1 - z)^2
							}
					 + 	\widetilde{O}\left( (1-x)^{1/2} (1 - y)^{-6} (1 - z)^{-6} \right)
				\right)
			\\
			& \hspace{ 10 mm } = 	
						\frac{
								4 q^4 s^6 y^2 z^2
							}{
								(1 - y)^2
								(1 - z)^2
							}
					 + 	\widetilde{O}\left( (1-x)^{1/2} (1 - y)^{-6} (1 - z)^{-6} \right)
			.
\end{split}
\label{twoTermT1expansion}
\end{align}

\noindent 
Using the above expansions in the identity of Proposition \ref{twoSharedRecursion} establishes \ref{claim T2 expansion}:
\begin{align*}
	T^2(x/r, y/2s, z/2s) 	
			& = 	
					2 \genFunFirstLeafShared(x/r, y/2s, z/2s) 
				+ 	ryz T^1(x/r, y/2s, z/2s), 
			\\
			& = 	
					\widetilde{O}\left( (1-y)^{-3}(1-z)^{-3} \right)
				+	\widetilde{O}\left( (1-y)^{-6}(1-z)^{-6} \right)
			\\
			& = 	
					\widetilde{O}\left( (1-y)^{-6}(1-z)^{-6} \right)
			.
\end{align*}

\noindent
Substituting (\ref{twoTermT1expansion}) and a variation of the expansion in Proposition \ref{scaledTstarProperties}(v),
\begin{align*}
	( \deriv{x}T^* )(x/r) 	
			& = 	
					( 2s )^{-1} \,
					\scaledTstar'( x )
			\\
			& = 	
					( 2s )^{-1}  \, 
					\left( 
							 qs \, 2^{-1} (1-x)^{-1/2} 
						+ 	u(x) 
					\right),
			\\
			& = 	
					q \, 4^{-1} (1-x)^{-1/2} + ( 2s )^{-1} u(x)
			,
\end{align*}

\noindent
into Proposition \ref{noneSharedIdentity}, we obtain our final expansion:
\begin{align*}
	T^0(x/r, y/2s, z/2s) 
			& = 	
					T^1(x/r, y/2s, z/2s) 
					( \deriv{x}T^* )(x/r) 
			\\
			& = 	
				\left(
						\frac{
								4
								q^4 s^6
								y^2 z^2
							}{
								(1 - y)^2
								(1 - z)^2
							}
					 + 	\widetilde{O}\left( (1-x)^{1/2} (1 - y)^{-6} (1 - z)^{-6} \right)
				 \right)
			\\
			& \quad \times
					\left(
							q 4^{-1} (1-x)^{-1/2} + ( 2s )^{-1} u(x)
					\right)
			\\
			& = 	
					q^5 s^6
					y^2 z^2
					\frac{
							(1-x)^{-1/2}
						}{
							(1 - y)^{2} 
							(1 - z)^{2}
						}
				+	\widetilde{O}\left( (1-y)^{-6}(1-z)^{-6} \right)
			.
\end{align*}

\noindent
To obtain \ref{claim T0 expansion}, we write
\begin{align*}
	C( y, z )
			& = 
					q^5 s^6 y^2 z^2
			\\
			& = 
					q^5 s^6 
				- 	q^5 s^6 ( 1 - y^2 ) 
				- 	q^5 s^6 y^2 ( 1 - z^2 )
			,
\end{align*}

\noindent
and observe that 
	$ C( 1, 1 ) \neq 0 $, 
	$ C( y, z ) = C( 1, 1 ) + O( | 1 - y | + | 1 - z | )$,  
	$ C( y, z ) $ is analytic in any product of $ \delta $-domains, and 
	$ q^4 s^4 = 2 $.
\end{proof}

\begin{prop}
\label{propTrivariateCoefficients}

The following statements hold:
\begin{enumerate}[label = (\roman*)]
	\item
	\label{claim trivariate expansion}
	$
		T(x/r, y/2s, z/2s) 	
				\approx 	 
						2qs^2 \, (1-x)^{-1/2}(1-y)^{-2} (1-z)^{-2}
	$, and

	\item
	\label{claim trivariate coefficient estimate}
    for every $ L > 0 $ there exists some $ M > 0 $ such that
	\begin{align*}
		\left|
			T_{n, k, k}
		-
			\frac{ 
        			2qs^2
				}{
					\sqrt{ \pi }
				}
			\,
			r^{ n - k }
			n^{ -1/2 }
			k^2
		\right|	
			\le
				M
				r^{ n - k }
				n^{ -1/2 }
        		k
	\end{align*}

	\noindent
	whenever $ k \le L \log n $ and $ k $ is large.

\end{enumerate}
\end{prop}

\begin{proof}

    Combining Proposition \ref{trivariateExpansions} with the relationship 
    $
    	T(x, y, z)
    		=
    				T^0(x, y, z)
    			+	T^1(x, y, z)
    			+	T^2(x, y, z)
    $
    establishes \ref{claim trivariate expansion}.
    Now we apply Proposition \ref{prop GW Lemma 3}.
    For $ L > 0 $, there exists some $ M > 0 $ such that
	$$
		\left|
			[ x^n y^k z^k ]
        	T(x/r, y/2s, z/2s)
		-
			2qs^2
			\,
			\frac{ 
					n^{ -1/2 }
				}{
					\Gamma( \frac{1}{2} )
				}
			\frac{ 
					k^2
				}{
					\Gamma( 2 )^2
				}
		\right|	
			\le
				M
				n^{ -1/2 }
        		k^2
        		\,
				\frac{2}{k}
	$$

	\noindent
	whenever $ k \le L \log n $ and $ k $ is large.
	Noting that
	$
    	T_{n, k, k} 	
    			= 	
    				r^n 
    				(2s)^{2k} 
    				[x^n y^k z^k] T(x/r, y/2s, z/2s) 
	$
	and
	$
		(2s)^{2k} 
			=
				r^{ -k }
	$
	concludes the proof.
\begin{align*}
	T_{n, k, k} 	
			& = 	
					r^n 
					(2s)^{2k} 
					[x^n y^k z^k] T(x/r, y/2s, z/2s) 
			\\
			& = 	
					r^n 
					(2s)^{2k} 
					\,
					2 q s^2 
					\frac{n^{-1/2}}{\Gamma(\frac{1}{2})} 
					\frac{k^2}{\Gamma(2)^2} 
					\left( 1 + O( 1/k ) \right) 
			\\
			& = 	
					r^{n-k} 
					\,
					2 q s^2 
					\pi^{-1/2} 
					n^{-1/2} 
					k^2 
					\left( 1 + O(1/k) \right) 
\end{align*}
\end{proof}

We now have the coefficient estimates we need to prove Lemma \ref{momentEstimates}.

\begin{proof}[Proof of Lemma \ref{momentEstimates}]	



Let $p \in (0, r^{1/2})$ and $ M $ be the constant obtained by applying Proposition \ref{propBivariateCoefficients}\ref{claim bivariate coefficient estimate all k}.
Combining the bound in that result with Proposition \ref{propUnivariateCoefficients}, 
we have that
\begin{align*}
	(n + 1) \frac{T_{n, k}}{T_n}
						& \le 	(n + 1) 
								\frac{  
										M n^{-3/2} p^{-k} r^n
									}{ 
										\frac{ q s^2 }{ 2 \sqrt{ \pi } } 
										n^{-3/2} r^n (1 + O(1/n))
								} 
						\\
						& \le 	M' n  p^{-k}
\end{align*}

\noindent 
for some $ M' > 0 $, large $ n $, and $ k \ge 2 $.
To obtain \ref{first moment estimate, all k} from this, 
	use the bijection between dissections and dual trees and Lemma \ref{momentsAndDissectionClasses}\ref{first moment identity} to identify the left hand side as $ \mbb{E}(\zeta_k) $ and observe that $ b = r^{1/2} $.

Now let $ L > 0 $ and $ M $ be the constant obtained by applying Proposition \ref{propBivariateCoefficients}\ref{claim bivariate coefficient estimate log k}.
Combining the bound in that result with Proposition \ref{propUnivariateCoefficients}, we find that
	$$
		\frac{
				( n + 1 ) b^k
			}{
				2 n
			}
		\left|
			\frac{ 
					T_{n, k}
				}{
					T_n
            	}
		-
			\frac{
        			\tfrac{ 
        					qs^2
        				}{
        					\sqrt{\pi}
        				}
					r^{n-k/2}
					n^{ -3/2 }
					k
				}{
					T_n
				}
		\right|	
				\le
						\frac{
							( n + 1 ) b^k 
							M
							r^{n-k/2}
        					n^{ -3/2 }
    					}{
							2 n \,
							\frac{ q s^2 }{ 2 \sqrt{ \pi } } 
    						r^n n^{-3/2} (1 + O(1/n))
    					}
	$$

	\noindent
	whenever $ k \le L \log n $ and $ k $ is large. 
    Making use of the identity $ b = r^{1/2} $, Lemma \ref{momentsAndDissectionClasses}\ref{first moment identity}, and the bijection between dissections and dual trees, this inequality implies that
    $$
    		\left|
    			\frac{b^k\mbb{E}(\zeta_k)}{2n}
    		-
    			\frac{
            			n + 1
					}{
						n + O(1)
					}
				\,
				k
    		\right|	
				\le
        				M'
    $$

	\noindent
	for some $ M' > 0 $ whenever $ k \le L \log n $ and $ k $ is large.
	Since
    $$
    		\left|
    			k
    		-
    			\frac{
            			n + 1
					}{
						n + O(1)
					}
				\,
				k
    		\right|	
    $$

	\noindent
	can also be bounded by a constant under the same conditions, this establishes \ref{first moment estimate, log k}.

	To establish \ref{second moment estimate}, we first use the bijection between dissections and dual trees and Lemma \ref{momentsAndDissectionClasses} to write
	$$
		k 
		\left|
    		\frac{ \mbb{E}(\zeta_k(\zeta_k - 1)) }
    		{\mbb{E}(\zeta_k)^{ 2 }}
		-
			1
		\right|
			=	
        		k
				\left|
            		\frac{ 1 }{ n + 1 }
					\frac{ 
							T_{ n, k, k }
							T_n
						}{
							(T_{n, k})^2
						}
        		-
        			1
        		\right|
			.
	$$
	
	\noindent
	Now let $ L > 0 $ and $ a $ be the sign of the above difference.
	Applying Propositions \ref{propUnivariateCoefficients}, \ref{propBivariateCoefficients}, and \ref{propTrivariateCoefficients}, we find that there exists some $ M > 0 $ such that
	\begin{align*}
   		&
		k
		\left|
        	\frac{ 1 }{ n + 1 }
			\frac{ 
					T_{ n, k, k }
					T_n
				}{
					(T_{n, k})^2
				}
    		-
    			1
    	\right|
		\\
			& \qquad 
			\le
				k
				a
        		\left(
                	\frac{ 1 }{ n + 1 }
        			\frac{ 
							\left(
                    			\tfrac{ 
                            			2qs^2
                    				}{
                    					\sqrt{ \pi }
                    				}
                    			\,
                    			r^{ n - k }
                    			n^{ -1/2 }
                    			k^2
    							( 1 + a \frac{ M }{ k }  )
							\right)
        					\left( \tfrac{qs^2}{2 \sqrt{\pi}}r^n n^{-3/2} ( 1 + a \frac{ M }{ n } ) \right)
        				}{
                            \left(
                    			\tfrac{ 
                    					qs^2
                        			}{
                        				\sqrt{\pi}
                        			}
                                r^{n-k/2}
                        		n^{ -3/2 }
                        		k
    							(
                        			1
    							-
                        			a \frac{ M }{ k } 
                        		)
                    		\right)^2
        				}
            		-
            			1
            	\right)
			\\
			& \qquad =
				k
				a
        		\left(
                	\frac{ n }{ n + 1 }
        			\frac{ 
							( 1 + a \frac{ M }{ k } )
        					( 1 + a \frac{ M }{ n } )
        				}{
							(
                    			1
							-
                    			a \frac{ M }{ k } 
                    		)^2
        				}
            		-
            			1
            	\right)
	\end{align*}

	\noindent
	whenever $ k \le L \log n $ and $ k $ is large.  This final quantity is easily seen to be bounded since $ k \le L \log n $, which concludes the proof.
\end{proof}

\section*{Acknowledgement}
This material is based upon work supported by the National Science Foundation under Grant No. DMS-1855568.

\bibliographystyle{plain}
\bibliography{Bibliography/bibPolygonDissections}

\end{document}